\documentclass[12pt]{amsart}
\usepackage{amsmath,amssymb,latexsym,verbatim}
\usepackage[mathscr]{euscript}
\title[Duality and distance formulas]{Duality and distance formulas in spaces defined by means of oscillation}
\author{Karl-Mikael Perfekt}

\date{\today}
\address{}
\email{}

\def\C{{\mathbb C}}
\def\R{{\mathbb R}}

\def\D{{\mathbb D}}
\def\L{{\mathcal L}}

\newcommand{\ud}{\mathrm{d}}
\newcommand{\norm}[1]{\left\Vert#1\right\Vert}
\newcommand{\dist}[1]{\mathsf{dist} ( #1 )}

\newtheorem{theorem}{Theorem}[section]
\newtheorem{proposition}[theorem]{Proposition}

\newtheorem{lemma}[theorem]{Lemma}
\newtheorem{corollary}[theorem]{Corollary}
\theoremstyle{definition}
\newtheorem{example}[theorem]{Example}

\theoremstyle{remark}
\newtheorem{remark}[theorem]{Remark}

\theoremstyle{definition}
\newtheorem{assumption}{Assumption}

\theoremstyle{definition}
\newtheorem*{claim}{Claim}

\begin{document}
\begin{abstract}
For the classical space of functions with bounded mean oscillation, it is well known that $VMO^{**} = BMO$ and there are many characterizations of the distance from a function $f$ in $BMO$ to $VMO$. When considering the Bloch space, results in the same vein are available with respect to the little Bloch space. In this paper such duality results and distance formulas are obtained by pure functional analysis. Applications include general M\"{o}bius invariant spaces such as $Q_K$-spaces, weighted spaces, Lipschitz-H\"{o}lder spaces and rectangular $BMO$ of several variables. 
\end{abstract}

\maketitle

\section{Introduction}
It is well known that the bidual of $VMO$ is $BMO$, that is, the second dual of the space of functions on the unit circle $\mathbb{T}$ (or the line $\R$) with vanishing mean oscillation can be naturally represented as the space of functions with bounded mean oscillation. The same holds true for the respective subspaces $VMOA$ and $BMOA$ of those functions in $VMO$ or $BMO$ whose harmonic extensions are analytic, and there has been considerable interest in estimating the distance from a function $f \in BMOA$ to $VMOA$, starting with Axler and Shapiro \cite{axlershapiro}, continuing with Carmona and Cuf\'{i} \cite{carmonacufi} and Stegenga and Stephenson \cite{stegengasteph}. For the Bloch space $B$ and little Bloch space $B_0$ the situation is similar. $B_0^{**} = B$ and the distance from $f \in B$ to $B_0$ has been characterized by Attele \cite{attele} and Tjani \cite{tjani}. 

Numerous people have explored the validity of the biduality $H v_0(\Omega)^{**} = H v(\Omega)$, where $H v(\Omega)$ is a weighted space, consisting of analytic functions bounded under the weighted supremum norm given by a weight function v on $\Omega \subset \C$, and $Hv_0(\Omega)$ denotes the corresponding little space (see Example \ref{weightedex} for details). For example, Rubel and Shields \cite{rubelshields} considered the problem when $\Omega = \D$ is the unit disc and $v$ is radial. Later, Anderson and Duncan \cite{andersonduncan} addressed the problem for $\Omega = \C$, and Bierstedt and Summers \cite{bierstedt} went on to give a characterization of the weights $v$ for which the biduality holds for general $\Omega$.
%The latter two authors remark that the biduality was also studied in an unpublished manuscript by Shapiro, Shields and Taylor \cite{shapiroshieldstaylor}.

The purpose of this paper is to obtain such duality results and distance formulas in a very general setting. Working with a Banach space $M$ defined by a big-O condition and a corresponding ``little space'' $M_0$, we will under mild assumptions show that $M_0^{**}=M$ and prove an isometric formula for the distance from $f \in M$ to $M_0$ in terms of the defining condition for $M_0$. Using a theorem of Godefroy \cite{godefroy}, we will additionally obtain as a corollary that $M_0^*$ is the unique isometric predual of $M$. When $M = B$ this specializes to a result of Nara \cite{nara}. The methods involved are purely operator-theoretic, appealing to embeddings into spaces of continuous vector-valued functions rather than analyticity, invariance properties or geometry. 

The power of these general results is illustrated in the final section. Many examples will be given there, where the main theorems are applied to general M\"{o}bius invariant spaces of analytic functions including a large class of so-called $Q_K$-spaces, weighted spaces, rectangular $BMO$ of several variables and Lipschitz-H\"{o}lder spaces.

A brief outline of our approach is given below. Fixing the notation for this, and for the rest of this paper, $X$ and $Y$ will be two Banach spaces, with $X$ separable and reflexive. $\L$ will be a given collection of bounded operators $L:X \to Y$ that is accompanied by a $\sigma$-compact locally compact Hausdorff topology $\tau$ such that for every $x \in X$, the map $T_x:\L \to Y$ given by $T_x L = Lx$ is continuous. Here $Y$ is considered with its norm topology. Note that we impose no particular algebraic structure on $\L$.  $Z^*$ will denote the dual of a Banach space $Z$ and we shall without mention identify $Z$ as a subset of $Z^{**}$ in the usual way.

Our main objects of study are the two spaces
\begin{equation*}
M(X,\L) = \left\{ x \in X \, : \, \sup_\L \norm{Lx}_Y < \infty \right\}
\end{equation*}
and
\begin{equation*}
M_0(X,\L) = \left\{ x \in M(X, \L) \, : \, \varlimsup_{\L \ni L\to\infty} \norm{Lx}_Y = 0 \right\},
\end{equation*}
where the limit $L \to \infty$ is taken in the sense of the one-point compactification of $(\L, \tau)$. By replacing $X$ with the closure of $M(X,\L)$ in $X$ and making appropriate modifications to the setting just described, we may as well assume that $M(X,\L)$ is dense in $X$. Furthermore, we assume that $\L$ is such that
\begin{equation*}
\norm{x}_{M(X,\L)} = \sup_\L \norm{Lx}_Y
\end{equation*}
defines a norm which makes $M(X,\L)$ into a Banach space continuously contained in $X$. Note that $M_0(X,\L)$ is then automatically a closed subspace of $M(X,\L)$. These assumptions are mostly for convenience and will hold trivially in all examples to come.

\begin{example} \label{bmoex}
Let $X = L^2(\mathbb{T})/\C$, $Y = L^1(\mathbb{T})$ and
\begin{equation*}
\L = \left\{ L_I \, : \, L_If = \chi_I \frac{1}{|I|}(f - f_I), \quad \emptyset \neq I \subset \mathbb{T} \,\, \text{is an arc} \right\},
\end{equation*}
where $\chi_I$ is the characteristic function of $I$, $|I|$ is its length and $f_I = \frac{1}{|I|} \int_I f \, \ud s$ is the average of $f$ on $I$. Each arc $I$ is given by its midpoint $a \in \mathbb{T}$ and length $b$, $0 < b \leq 2\pi$. We give $\L$ the quotient topology $\tau$ of $\mathbb{T} \times (0, 2 \pi]$ obtained when identifying all pairs $(a_1, b_1)$ and $(a_2,b_2)$ with $b_1 = b_2 = 2\pi$. Then $M(X, \L) = BMO(\mathbb{T})$ is the space of functions of bounded mean oscillation on the circle. $L_I \to \infty$ in $\tau$ means exactly that $|I|\to 0$, so it follows that $M_0(X, \L) = VMO(\mathbb{T})$ are the functions of vanishing mean oscillation (see Garnett \cite{garnett}, Chapter VI). 
\end{example}

$M_0(X, \L)$ may be trivial even when $M(X,\L)$ is not. This happens for example when $M(X,\L)$ is the space of Lipschitz-continuous functions $f$ on $[0,1]$ with $f(0) = 0$ (see Example \ref{lipex}). In the general context considered here we shall not say anything about this, but instead make one of the following two assumptions. They say that $M_0(X, \L)$ is dense in $X$ (under the $X$-norm) with additional norm control when approximating elements of $M(X,\L)$. This is a natural hypothesis that is easy to verify in the examples we have in mind. In fact, the assumptions are necessary for the respective conclusions of the main theorems.

\begin{assumption} \label{as1}For every $x\in M(X, \L)$ there is a sequence $(x_n)$ in $M_0(X,\L)$ such that $x_n \to x$ in $X$ and $\sup_n \norm{x_n}_{M(X,\L)} < \infty$. 
\end{assumption}

\begin{assumption} \label{as2}
For every $x\in M(X,\L)$ there is a sequence $(x_n)$ in $M_0(X,\L)$ such that $x_n \to x$ in $X$ and $\sup_n \norm{x_n}_{M(X,\L)} \leq \norm{x}_{M(X,\L)}$. 
\end{assumption}

Note that the assumptions could have equivalently been stated with the sequence $(x_n)$ tending to $x$ only weakly in $X$. The main theorems are as follows.

\begin{theorem} \label{thm1}
Suppose that assumption \ref{as1} holds. Then $X^*$ is continuously contained and dense in $M_0(X, \L)^*$. Denoting by $$I : M_0(X, \L)^{**} \to X$$ the adjoint of the inclusion map $J : X^* \to M_0(X,\L)^*$, the operator $I$ is a continuous isomorphism of $M_0(X, \L)^{**}$ onto $M(X,\L)$ which acts as the identity on $M_0(X,\L)$. Furthermore, $I$ is an isometry if assumption \ref{as2} holds. 
\end{theorem}

\begin{theorem} \label{thm2}
Let assumption \ref{as1} hold. Then, for any $x \in M(X,\L)$, it holds that
\begin{equation} \label{thm2equation}
\dist{x, M_0(X, \L)}_{M(X,\L)} = \varlimsup_{\L \ni L\to\infty} \norm{Lx}_Y.
\end{equation}
\end{theorem}

\begin{example} \label{bmoexpt2}
Let $X$, $Y$, $\L$ and $\tau$ be as in Example \ref{bmoex}. Then Assumption \ref{as2} holds by letting $f_n = f * P_{1-1/n}$ for $f \in M(X,\L)$, where $P_r$ is the Poisson kernel for the unit disc, $P_r(\theta) = \frac{1-r^2}{|e^{i\theta} - r|^2}$. The theorems say that $VMO(\mathbb{T})^{**} \simeq BMO(\mathbb{T})$ isometrically via the $L^2(\mathbb{T})$-pairing, and that 
\begin{equation} \label{bmoexpt2formula}
\dist{f, VMO}_{BMO} = \varlimsup_{|I| \to 0} \frac{1}{|I|} \int_I |f-f_I| \, \ud s.
\end{equation}
This improves upon a result in \cite{stegengasteph}. Note that if we repeat the construction with $Y = L^p(\mathbb{T})$ for some $1 < p < \infty$, we still obtain that $M(X,\L) = BMO(\mathbb{T})$ with an equivalent norm, due to the John-Nirenberg theorem (see \cite{garnett}). This gives us a distance formula, corresponding to \eqref{bmoexpt2formula}, involving the $p$-norm on the right-hand side.
\end{example}

 We say that $Z$ is a unique (isometric) predual if for any Banach space $W$, $W^*$ isometric to $Z^*$ implies $W$ isometric to $Z$. Note that the canonical decomposition 
\begin{equation} \label{decomp}
Z^{***} = Z^* \oplus Z^{\perp}
\end{equation}
 induces a projection $\pi : Z^{***} \to Z^*$ with kernel $Z^{\perp}$, 
\begin{equation*}
(\pi z^{***})(z) = z^{***}(z), \quad \forall z \in Z. 
\end{equation*}
We say that $Z$ is a strongly unique predual if this is the only projection $\pi$ from $Z^{***}$ to $Z^*$ of norm one with $\text{Ker} \, \pi$ weak-star closed. An excellent survey of these matters can be found in Godefroy \cite{godefroy}.

\begin{corollary} \label{maincor}
Suppose that Assumption \ref{as2} holds. Then $M_0(X,\L)^*$ is the strongly unique predual of $M(X, \L)$.
\end{corollary}

\section{Preliminaries}
One of our main tools will be the isometric embedding $V: M(X, \L) \to C_b(\L, Y)$ of $M(X, \L)$ into the space $C_b(\L, Y)$ of bounded continuous $Y$-valued functions on $\L$, given by
\begin{equation*}
Vx = T_x,
\end{equation*}
where $T_xL = Lx$ as before. $C_b(\L, Y)$ is normed by the usual supremum-norm, so that $V$ indeed is an isometry. Note that $V$ embeds $M_0(X, \L)$ into the space $C_0(\L, Y)$, consisting of those functions $T \in C_b(\L, Y)$ vanishing at infinity.

In order to study duality via this embedding, we will make use of vector-valued integration theory. Of central importance will be the Riesz-Zinger theorem \cite{zinger} for $C_0(\L, Y)$, representing the dual of $C_0(\L, Y)$ as a space of measures. Let $\mathscr{B}_0$ be the $\sigma$-algebra of all Baire sets of $\L$, generated by the compact $G_\delta$-sets. By $cabv(\L, Y^*)$ we shall denote the Banach space of countably additive vector Baire measures $\mu : \mathscr{B}_0 \to Y^*$ with bounded variation
\begin{equation*}
\norm{\mu} = \sup \sum \norm{\mu(\mathcal{E}_i)}_{Y^*} < \infty,
\end{equation*} 
where the supremum is taken over all finite partitions of $\L = \cup \mathcal{E}_i$ into disjoint Baire sets $\mathcal{E}_i$. Excellent references for these matters are found for example in Dobrakov \cite{dobrakov1}, \cite{dobrakov2} and \cite{dobrakov3}. 

\begin{theorem}[\cite{dobrakov3}, Theorem 2] \label{zingerthm}
For every bounded linear functional $\ell \in C_0(\L, Y)^*$ there is a unique measure $\mu \in cabv(\L, Y^*)$ such that
\begin{equation} \label{zingereq}
\ell(T) = \int_\L T(L) \, \ud \mu(L), \quad T \in C_0(\L, Y).
\end{equation}
Furthermore, every $\mu \in cabv(\L, Y^*)$ defines a continuous functional on $C_0(\L, Y)$ via \eqref{zingereq}, and $\norm{\ell} = \norm{\mu}$. 
\end{theorem}
\begin{remark} \label{zingerthmrmk}
In our situation, $(\L, \tau)$ being $\sigma$-compact, every continuous function $T:\L \to Y$ is Baire measurable (\cite{halmos}, pp. 220--221). In particular, every $T \in C_b(\L, Y)$ induces a bounded functional $m \in cabv(\L, Y^*)^*$ via
\begin{equation*}
m(\mu) = \int_\L T(L) \, \ud \mu(L).
\end{equation*}
It is clear that $\norm{m} = \norm{T}_{C_b(\L, Y)}$. Hence, $C_b(\L,Y)$ isometrically embeds into $cabv(\L, Y^*)^*$ in a way that extends the canonical embedding of $C_0(\L, Y)$ into $C_0(\L,Y)^{**} = cabv(\L, Y^*)^*$. 
\end{remark}

In the proof of Theorem \ref{thm2} we shall require the following simple lemma.
\begin{lemma} \label{distlemma}
Suppose that $m \in C_b(\L,Y)^*$ annihilates $C_0(\L,Y)$. Then
\begin{equation} \label{distlemmaeq}
|m(T)| \leq \norm{m} \varlimsup_{L \to \infty} \norm{T(L)}_Y, \quad T \in C_b(\L,Y).
\end{equation}
\end{lemma}
\begin{proof}
Let $\mathcal{K}_1 \subset \mathcal{K}_2 \subset \cdots$ be an increasing sequence of compact subsets of $(\L, \tau)$ such that $\L = \cup \mathcal{K}_n$. Denote by $\alpha \L = \L \cup \{\infty\}$ the one point compactification of $\L$. For each $n$, let $s_n : \alpha L \to [0,1]$ be a continuous function such that $s_n^{-1}(0) \supset \mathcal{K}_n$ and $s_n(\infty) = 1$. Then
\begin{equation*}
|m(T)| = |m(s_nT)| \leq \norm{m} \sup_{L \in \L \setminus \mathcal{K}_n} \norm{T(L)}_Y.
\end{equation*}
In the limit we obtain \eqref{distlemmaeq}.
\end{proof}

Corollary \ref{maincor} will follow as an application of a result in \cite{godefroy}.
\begin{theorem}[\cite{godefroy}, Theorem V.1] \label{godefroythm}
Let $Z$ be a Banach space. Suppose that for every $z^{**} \in Z^{**}$ it holds that: $z^{**} \in Z$ if and only if $$z^{**}(z^*) = \lim_n z^{**}(z^*_n)$$ for every weak Cauchy sequence $(z^*_n)$ in $Z^*$ with weak-star limit $z^*$. Then $Z$ is the strongly unique predual of $Z^*$.
\end{theorem}

\section{Main Results}
We shall now proceed to prove the main theorems and Corollary \ref{maincor}.
\begin{proof}[Proof of Theorem \ref{thm2}]
Seeing as $M_0(X,\L)$ is continuously contained in $X$, every $x^* \in X^*$ is clearly continuous also on $M_0(X,\L)$. Assumption \ref{as1} implies that $M_0(X,\L)$ is dense in $X$, so that each element of $X^*$ induces a unique functional on $M_0(X,\L)$. This proves that $X^*$ is continuously contained in $M_0(X,\L)^*$. 

We shall now demonstrate that $X^*$ is dense in $M_0(X,\L)^*$. Thus, let $\ell \in M_0(X,\L)^*$. By Theorem \ref{zingerthm} and the Hahn-Banach theorem there is a measure $\mu \in cabv(\L, Y^*)$ such that
\begin{equation*}
\ell(x) = (\ell \circ V^{-1}) (T_x) = \int_\L Lx \, \ud\mu(L), \quad x \in M_0(X,\L). 
\end{equation*}
Let $\mathcal{K}_1 \subset \mathcal{K}_2 \subset \cdots$ be an increasing sequence of compact subsets of $(\L, \tau)$ such that $\L = \cup \mathcal{K}_n$. By (\cite{halmos}, 50.D) we can choose the $\mathcal{K}_n$ to be $G_\delta$ and hence Baire sets. Put $\mu_n = \mu|\mathcal{K}_n$ and let $\ell_n$ be the corresponding functionals
\begin{equation*}                                                               \ell_n(x) = \int_\L Lx \, \ud\mu_n(L).
\end{equation*}
The operators $L \in \mathcal{K}_n$ are uniformly bounded by the Banach-Steinhaus theorem, from which it is clear that $\ell_n \in X^*$. Finally, note that $\lim_n \norm{\mu_n - \mu} = 0$, which implies that the functionals $\ell_n$ converge to $\ell$ in $M_0(X,\L)^*$.

$X^*$ being dense ensures the 
%continuity and 
injectivity of $I = J^*$. Moreover, it is clear that $I$ acts as the identity on $M_0(X,\L)$. Let $m \in M_0(X,\L)^{**}$ and $x = Im \in X$. Note that the unit ball of $M_0(X,\L)$ is weak-star dense in the unit ball of $M_0(X,\L)^{**}$ (\cite{conway}, Proposition 4.1). Furthermore, the weak-star topology of $M_0(X,\L)^{**}$ is metrizable on the unit ball, since $M_0(X,\L)^{*}$ was just proven to be separable. Accordingly, choose a sequence $x_n \in M_0(X,\L)$ with $\norm{x_n} \leq \norm{m}$ such that $x_n \to m$ weak-star. Then, for $y^* \in Y^*$ and $L \in \L$,
\begin{align*}
y^*(Lx) &= (L^*y^*)(x) = m(JL^*y^*) = \lim_n (JL^*y^*)(x_n) \\ 
        &= \lim_n (L^*y^*)(x_n) = \lim_n y^*(Lx_n).  
\end{align*}   
It follows that $x \in M(X, \L)$ and
\begin{equation} \label{Icontr}
\norm{x}_{M(X,\L)} = \norm{Im}_{M(X,\L)} \leq \norm{m}_{M_0(X,\L)^{**}},
\end{equation}
since
\begin{equation*}
\norm{Lx}_Y = \sup_{\norm{y^*}=1} |y^*(Lx)| \leq \sup_n \norm{x_n}_{M(X, \L)} \leq \norm{m}, \quad \forall L \in \L.
\end{equation*}
We have thus proved that $I$ maps $M_0(X,\L)^{**}$ into $M(X, \L)$ contractively.

Given $x \in M(X, \L)$ choose $x_n \in M_0(X,\L)$ such that $x_n \to x$ in $X$ and $\sup_n \norm{x_n}_{M(X, \L)} < \infty$ ($\leq \norm{x}_{M(X, \L)}$ if assumption \ref{as2} holds). Define $\hat{x} \in M_0(X,\L)^{**}$ by $\hat{x}(Jx^*) = x^*(x) = \lim_n (Jx^*)(x_n)$ for $x^* \in X^*$. It is clear from the last equality that this defines $\hat{x}$ as a bounded linear functional on $ M_0(X,\L)^*$ and if assumption \ref{as2} holds, then 
\begin{equation} \label{Iexp}
\norm{\hat{x}}_{M_0(X,\L)^{**}} \leq \norm{x}_{M(X,\L)}.
\end{equation}
Obviously, $I\hat{x} = x$. This proves that $I$ is onto. If assumption \ref{as2} holds, then we obtain from \eqref{Icontr} and \eqref{Iexp} that $I$ is an isometry.
\end{proof}

\begin{proof}[Proof of Theorem \ref{thm2}]
Let $m \in M(X,\L)^*$. Then $m \circ V^{-1}$ acts on $VM(X,\L)$. As in Remark \ref{zingerthmrmk} we naturally view $C_b(\L,Y)$ as a subspace of $cabv(\L,Y^*)^*$. With this identification, $m \circ V^{-1}$ extends by Hahn-Banach to a functional $\bar{m} \in cabv(\L,Y^*)^{**}$ with $\norm{\bar{m}} = \norm{m}$. Applying the decomposition \eqref{decomp} with $Z = C_0(\L, Y)$ we obtain
\begin{equation*}
cabv(\L,Y^*)^{**} = cabv(\L,Y^*) \oplus C_0(\L,Y)^\perp,
\end{equation*}
and we decompose $\bar{m} = \bar{m}_{\omega^*} + \bar{m}_s$ accordingly. Let $\mu \in cabv(\L,Y^*)$ be the measure corresponding to $\bar{m}_{\omega^*}$, so that, in particular, 
\begin{equation*}
\bar{m}_{\omega^*}(T) = \int_\L T(L) \, \ud \mu(L), \quad T \in C_b(\L,Y).
\end{equation*}

Let $I:M_0(X,\L)^{**} \to M(X,\L)$ be the isomorphism given by Theorem \ref{thm1}. With $Z = M_0(X,\L)$, \eqref{decomp} gives
\begin{equation} \label{mxdualdecomp} 
M(X,\L)^* \simeq M_0(X,\L)^{***} = M_0(X,\L)^* \oplus M_0(X,\L)^\perp ,
\end{equation}
and we obtain a second decomposition $m \circ I = (m \circ I)_{\omega^*} + (m \circ I)_s$.

Our first goal is to show that the former decomposition is an extension of the latter. More precisely, we have
\begin{claim}
$(m \circ I)_{\omega^*} \equiv 0$ if and only if $\bar{m}_{\omega^*}$ annihilates $VM(X,\L)$.
\end{claim}
To prove this, let $x \in M(X,\L)$ and let $\hat{x} = I^{-1}x \in M_0(X,\L)^{**}$. As in the proof of Theorem \ref{thm1}, choose $x_n \in M_0(X,\L)$ with $\norm{x_n}_{M_0(X,\L)} \leq \norm{\hat{x}}$ such that $x_n \to \hat{x}$ weak-star. Note that $(x_n)$ in particular converges to $x$ weakly in $X$. Hence $Lx_n \to Lx$ weakly in $Z$ for every $L \in \L$. Since also $\sup_{n,L} \norm{Lx_n}_Y < \infty$, it follows from (\cite{dobrakov3}, Theorem 9) that
\begin{equation} \label{weakstarconveq}
\int_\L Lx \, \ud \mu(L) = \lim_n \int_\L Lx_n \, \ud \mu(L).
\end{equation}
To be more precise, Theorem 9 in \cite{dobrakov3} allows us to move the limit inside the integral when integrating over a compact $G_\delta$-set $\mathcal{K} \subset \L$. However, as in the proof of Theorem \ref{thm1}, we obtain \eqref{weakstarconveq} by an obvious approximation argument. We thus have 
\begin{multline*}
\bar{m}_{\omega^*}(Vx) = \lim_n \int_\L Lx_n \, \ud \mu(L) = \lim_n \bar{m}_{\omega^*}(Vx_n) = \lim_n \bar{m}(Vx_n)  = \\ \lim_n m(x_n) = \lim_n \, (m\circ I)(x_n) = \lim_n \, (m \circ I)_{\omega^*}(x_n) = (m \circ I)_{\omega^*}(\hat{x}),
\end{multline*} 
so that the claim is proven. 

We can now calculate the distance from $x \in M(X,\L)$ to $M_0(X,\L)$ using duality.
\begin{equation*}
\dist{x, M_0(X, \L)}_{M(X,\L)} = \sup_{\substack{\norm{m}=1 \\ (m \circ I)_{\omega^*} \equiv 0}} |m(x)| = \sup_{\substack{\norm{m}=1 \\ \bar{m}_{\omega^*} \perp VM(X,\L)}} |\bar{m}_s(Vx)|.
\end{equation*}
Since $\norm{\bar{m}_s} \leq \norm{m}$ we obtain by Lemma \ref{distlemma} that
\begin{equation*}
\dist{x, M_0(X, \L)}_{M(X,\L)} \leq \varlimsup_{L\to\infty} \norm{Lx}_Y.
\end{equation*}
The converse inequality is trivial; for any $x_0 \in M_0(X, \L)$ we have
\begin{align*}
\norm{x-x_0}_{M(X,\L)} &\geq \varlimsup_{L\to\infty} \norm{Lx-Lx_0}_Y 
\\ &\geq \varlimsup_{L\to\infty} \left( \norm{Lx}_Y - \norm{Lx_0}_Y \right) = \varlimsup_{L\to\infty} \norm{Lx}_Y. \qedhere
\end{align*}
\end{proof}
\begin{proof}[Proof of Corollary \ref{maincor}]
As in the preceding proof, for $m \in M_0(X,\L)^{***}$, write 
\begin{equation*}
m = m_{\omega^*} + m_s,
\end{equation*}
in accordance with \eqref{mxdualdecomp}. Suppose that $m \notin M_0(X,\L)^*$, or equivalently, $m_s \neq 0$.  Pick $\hat{x} \in M_0(X,\L)^{**}$ such that $m_s(\hat{x}) \neq 0$ and let $x_n \in M_0(X,\L)$ converge to $\hat{x}$ weak-star. Then $(x_n)$, as a sequence in $M_0(X,\L)^{**}$, is weakly Cauchy, since 
\begin{equation*}
\lim_n \, m'(x_n) =  m'_{\omega^*}(\hat{x}),  \quad \forall m' \in M_0(X,\L)^{***}.
\end{equation*}
On the other hand,
\begin{equation*}
m(\hat{x}) = m_{\omega^*}(\hat{x}) + m_s(\hat{x}) \neq m_{\omega^*}(\hat{x}),
\end{equation*}
so that $$m(\hat{x}) \neq \lim_n m(x_n).$$ We have thus verified the condition of Theorem \ref{godefroythm} for $Z = M_0(X,\L)^*$, proving the corollary. 
\end{proof}

\section{Examples}
\begin{example} \label{blochex}
Denoting by $L^2_a = L^2(\D) \cap \text{Hol}(\D)$ the usual Bergman space on the unit disc $\D$, let $X = L^2_a/\C$ be the space of functions $f \in L^2_a$ with $f(0) = 0$. Let $Y = \C$,
\begin{equation*}
\L = \{L_w \, : \, L_wf = (1-|w|^2)f'(w), \quad w \in \D\},    
\end{equation*}
and let $\tau$ be the topology of $\D$. Then $M(X,\L) = B/\C$ is the Bloch space modulo constants and $M_0(X,\L) = B_0/\C$ is the little Bloch space (up to constants). For $f \in B$ it is clear that the dilations $f_r$, $f_r(z) = f(rz)$, converge to $f$ in $L^2_a$ as $r\to1^-$ and that $\norm{f_r}_B \leq \norm{f}_B$, verifying the hypothesis of Assumption \ref{as2}. From the theorems we obtain that $(B_0/\C)^{**} \simeq B/\C$ isometrically via the $L^2_a$-pairing, as well as the distance formula
\begin{equation*}
\dist{f, B_0/\C}_{B/\C} = \varlimsup_{|w|\to 1} (1-|w|^2)|f'(w)|.
\end{equation*}
This improves a result previously obtained in \cite{attele} and \cite{tjani}. Corollary \ref{maincor} says furthermore that the the Bloch space has a unique predual, reproducing a result found in \cite{nara}.
\end{example}

\begin{example}
Much more generally, the main theorems can be applied to M\"{o}bius invariant spaces of analytic functions through the following construction. Denote by $G$ the M\"{o}bius group, consisting of the conformal disc automorphisms $\phi : \D \to \D$. Each function in $G$ is of the form 
\begin{equation*}
\phi_{a, \lambda}(z) = \lambda \frac{a-z}{1-\bar{a}z}, \quad a \in \D, \, \lambda \in \mathbb{T}.
\end{equation*} 
$G$ is a topological group with the topology of $\D \times \mathbb{T}$. In particular, $\phi_{a, \lambda} \to \infty$ equivalently means that $|a| \to 1$.

Let $X$ be a Banach space whose members $f \in X$ are functions analytic in $\D$ with $f(0) = 0$. We assume that $X$ is continuously contained in $\text{Hol}(\D)/\C$, the latter space equipped with the compact-open topology, and that it satisfies the properties:
\begin{enumerate}
\renewcommand{\theenumi}{[\Roman{enumi}]}
\renewcommand{\labelenumi}{\textbf{\theenumi}}
\item $X$ is reflexive.
\item \label{property2} The holomorphic polynomials $p$ with $p(0) = 0$ are contained and dense in $X$.
\item \label{property3} For each fixed $f \in X$, the map $T_f$, $T_f \phi = f \circ \phi - f(\phi(0))$ is a continuous map from $G$ to $X$. 
\item \label{property4} $\varlimsup_{G \ni \phi \to \infty} \norm{\phi - \phi(0)}_X = 0$.
\end{enumerate}
We now let $Y = X$ and let $\L$ be the collection of composition operators induced by $G$,
\begin{equation*}
\L = \{ L_\phi \, : \, L_\phi f = f \circ \phi - f(\phi(0)), \quad \phi \in G \},
\end{equation*}
equipping it with the topology of $G$.

$M(X, \L)$ and $M_0(X, \L)$ are then M\"{o}bius invariant Banach spaces in the sense that if $f$ is in either space then so is $f\circ \phi - f(\phi(0))$, $\phi \in G$, and furthermore $\norm{f \circ \phi-f(\phi(0))}_{M(X,\L)} = \norm{f}_{M(X,\L)}$. Property \ref{property4} implies that $B_1/\C$ is continuously contained in $M_0(X,\L)$, where $B_1$ denotes the analytic Besov 1-space, the minimal M\"{o}bius invariant space (see \cite{AFP}). In particular $M_0(X,\L)$ contains the polynomials. The construction of the space $M(X, \L)$ has been considered by Aleman and Simbotin in \cite{AS}. 

General M\"{o}bius invariant spaces are studied by Arazy, Fisher and Peetre in \cite{AFP}. The next proposition, saying that Assumption \ref{as2} holds, is essentially contained there. Arazy et al. have a stricter definition of what a M\"{o}bius invariant space is, however, so its proof is included here for completeness.
\begin{proposition}
Given $f \in M(X, \L)$, $f(z) = \sum_1^\infty a_k z^k$, let
\begin{equation*}
f_n(z) = \sum_{k=1}^n \left( 1 - \frac{k}{n+1} \right) a_k z^k.
\end{equation*}
Then $\norm{f_n}_{M(X,\L)} \leq \norm{f}_{M(X,\L)}$ and $f_n \to f$ weakly in $X$.
\end{proposition}
\begin{proof}
Denote by $\Phi_n(\theta) = \sum_{k=-n}^n \left( 1 - \frac{|k|}{n+1} \right)e^{-ik\theta}$ the Fej\'{e}r kernels. That $\norm{f_n}_{M(X,\L)} \leq \norm{f}_{M(X,\L)}$ is immediate from the formula
\begin{equation*}
f_n \circ \phi - f_n(\phi(0)) = \frac{1}{2\pi} \int_0^{2\pi} \left( f(e^{i\theta}\phi(\, \cdot \,)) - f(e^{i\theta}\phi(0)) \right) \Phi_n(\theta) \, \ud \theta, 
\end{equation*}
where the the integral is to be understood as an $X$-valued function of $\theta$ integrated against the measure $\Phi_n(\theta) \, \ud \theta$. That $f_n \to f$ weakly follows from the same formula with $\phi(z) = z$, because if $\ell \in X^*$, then
\begin{equation*}
\ell(f_n) = \frac{1}{2\pi} \int_0^{2\pi} \ell( f(e^{i\theta} \, \cdot \,) ) \Phi_n(\theta)  \, \ud \theta \to \ell(f)
\end{equation*}
as $n \to \infty$, by a standard argument about the Fej\'{e}r kernels.
\end{proof} 
Applying the theorems, we obtain that $M_0(X,\L)^{**} \simeq M(X,\L)$ isometrically, that $M_0(X,\L)^*$ is the unique isometric predual of $M(X,\L)$ and the formula 
\begin{equation*}
\dist{f,M_0(X,\L)}_{M(X,\L)} = \varlimsup_{|a| \to 1} \norm{f \circ \phi_{a,\lambda} - f(\phi_{a,\lambda}(0))}_X.
\end{equation*}
There are many examples of M\"{o}bius invariant spaces. Letting $X = L^2_a/\C$ we once again obtain the Bloch space, $M(X,\L) = B/\C$ and $M_0(X,\L) = B_0/\C$, but with a different norm than in Example \ref{blochex}. When $X = H^2/\C$ is the Hardy space modulo constants we get the space of analytic $BMO$ functions with its conformally invariant norm, $M(X,\L) = BMOA/\C$ and $M_0(X,\L) = VMOA/\C$ (see \cite{garnett}).  

The $Q_K$-spaces provide a wide class of M\"{o}bius invariant spaces that includes both $B$ and $BMOA$. For a non-zero, right-continuous, non-decreasing function $K: [0,\infty) \to [0,\infty)$, denote by $Q_K$ the space of all $f \in \text{Hol}(\D)$ with $f(0) = 0$ such that
\begin{equation*}
\norm{f}_{Q_K}^2 = \sup_{\phi \in G} \int_\D |f'(z)|^2 K \left( \log \frac{1}{|\phi|} \right) \, \ud A(z) < \infty. 
\end{equation*}
See Ess\'{e}n and Wulan \cite{EssenWulan} for a survey of $Q_K$-spaces. See also \cite{AS}. Clearly, $Q_K = M(X_K, \L)$, if $X_K$ is the space of all $f \in \text{Hol}(\D)$ with $f(0) = 0$ such that
\begin{equation*}
\norm{f}_{X_K}^2 = \int_\D |f'(z)|^2 K \left( \log \frac{1}{|z|} \right) \, \ud A(z) < \infty.
\end{equation*}
If $K \left( \log \frac{1}{|z|} \right)$ is integrable on $\D$ and $K(\rho) \to 0$ as $\rho \to 0^+$, it is easy to verify that $X_K$ is a Hilbert space for which properties \ref{property2} and \ref{property4} hold. Furthermore, if $K \left( \log \frac{1}{|z|} \right)$ is a normal weight in the sense of Shields and Williams \cite{ShieldsWilliams}, standard arguments show that if $\phi_{a,\lambda} \to \phi$ in $G$, then
\begin{equation*}
\norm{f \circ \phi_{a,\lambda} - f(\phi_{a,\lambda}(0))}_{X_K} \to \norm{f \circ \phi - f(\phi(0))}_{X_K}, \quad f \in X_K.
\end{equation*}
Since $f \circ \phi_{a,\lambda} - f(\phi_{a,\lambda}(0))$ also tends weakly to $f \circ \phi - f(\phi(0))$, we in fact have norm convergence, verifying \ref{property3} under these assumptions. Hence, if we denote by $Q_{K,0} = M_0(X_K, \L)$ the space of those functions $f \in Q_K$ such that 
\begin{equation*}
\varlimsup_{|a|\to1} \int_\D |f'(z)|^2 K \left( \log \frac{1}{|\phi_{a,\lambda}|} \right) \, \ud A(z)= 0,
\end{equation*}
we have proven that $Q_{K,0}^{**} = Q_K$ and that 
\begin{equation*}
\dist{f,Q_{K,0}}_{Q_K} = \varlimsup_{|a|\to1} \int_\D |f'(z)|^2 K \left( \log \frac{1}{|\phi_{a,\lambda}|} \right) \, \ud A(z).
\end{equation*}

\end{example}

\begin{example} \label{weightedex}
For an open subset $\Omega$ of $\C$, let $v$ be a strictly positive continuous function on $\Omega$. In this example we shall consider the weighted space $H v(\Omega)$ of analytic functions on $\Omega$ bounded under the weighted supremum norm given by the weight $v$. 

For the purpose of applying our construction, choose an auxiliary strictly positive continuous weight function $w:\Omega \to \R_+$ such that $w$ is integrable on $\Omega$. Define $X = L^2_a(v^2w)$ to be the weighted Bergman space on $\Omega$ with weight $v^2w$, consisting of those $f \in \text{Hol}(\Omega)$ such that
\begin{equation*}
\norm{f}_{L^2_a(v^2w)} = \int_\Omega |f(z)|^2 v(z)^2w(z) \, \ud A(z) < \infty.
\end{equation*}
One easily verifies that $X$ is a Hilbert space continuously contained in $\text{Hol}(\D)$. Furthermore, let $Y = \C$, 
\begin{equation*}
\L = \{L_z \, : \, L_zf = v(z)f(z), \quad z \in \Omega \},    
\end{equation*}
and let $\tau$ be given by the usual topology of $\Omega$.

It is then clear that $M(X,\L) = H v(\Omega)$ is the Banach space of all $f \in \text{Hol}(\Omega)$ such that $vf$ is bounded, and that $M_0(X,\L) = H v_0(\Omega)$ is the corresponding little space, consisting of those $f$ such that $vf$ vanishes at infinity on $\Omega$. 

Assumption \ref{as1} (Assumption \ref{as2}) holds if and only if there for each $f \in H v(\Omega)$ exists a sequence $(f_n) \subset H v_0(\Omega)$ such that $f_n \to f$ pointwise in $\Omega$ and $\sup \norm{f_n}_{H v(\Omega)} < \infty$ ($\sup \norm{f_n}_{H v(\Omega)} \leq \norm{f}_{H v(\Omega)}$). We have hence recovered a result of Bierstedt and Summers \cite{bierstedt}; $H v_0(\Omega)^{**} \simeq H v(\Omega)$ via the natural isomorphism if and only if this pointwise weighted approximation condition holds. The isomorphism is isometric precisely when Assumption \ref{as2} holds. 

When either assumption holds we furthermore obtain the distance formula
\begin{equation*}
\dist{f, H v_0(\Omega)}_{H v(\Omega)} = \varlimsup_{ \substack{ \Omega \ni z \to \infty \\ \text{w.r.t. } \Omega}} v(z)|f(z)|.
\end{equation*}
Bierstedt and Summers give sufficient conditions for radial weights $v$ which ensure that Assumption \ref{as2} holds. For example, it holds when $v$ is a radial weight vanishing at $\partial \Omega$, where $\Omega$ is a balanced domain such that $\overline{\Omega}$ is a compact subset of $\{ z \in \C \, : \, rz \in \Omega\}$ for every $0 < r < 1$. When $\Omega = \C$, we may take any radial weight $v$ on $\C$ decreasing rapidly at infinity to obtain a space $H v(\C)$ of entire functions satisfying Assumption \ref{as2}. See \cite{bierstedt} for details.

For simplicity the above considerations have not been carried out in their full generality. We could, for example, have considered weighted spaces $M(X,\L)$ of harmonic functions, or of functions defined on $\C^n$, $n > 1$. However, problems arise when $\Omega$ is an open subset of an infinite-dimensional Banach space, since local compactness of $\Omega$ is lost. This case has been considered by Garc\'{i}a, Maestre and Rueda in \cite{garcia}. In a different direction, the biduality problem has been studied for weighted inductive limits of spaces of analytic functions. See Bierstedt, Bonet and Galbis \cite{bierstedtbonet} for results in this context.
\end{example}

\begin{example}
We now turn to rectangular bounded mean oscillation on the 2-torus. The space $BMO_{Rect}(\mathbb{T}^2)$ consists of those $f \in L^2(\mathbb{T}^2)/\C$ such that
\begin{equation*}
\sup \frac{1}{|I| |J|} \int_{I} \int_J |f(\zeta, \lambda) - f_J(\zeta) - f_I(\lambda) + f_{I\times J}|^2 \, \ud s(\zeta) \ud s(\lambda) < \infty,
\end{equation*}
where the supremum is taken over all subarcs $I, J \subset \mathbb{T}$, $\ud s$ is arc length measure, $\phi_J(\zeta) = \frac{1}{|J|} \int_J f(\zeta, \lambda) \, \ud s(\lambda)$ and $\phi_I(\lambda) = \frac{1}{|I|} \int_I f(\zeta, \lambda) \, \ud s(\zeta)$ are the averages of $f(\zeta, \, \cdot)$ and $f(\cdot, \lambda)$ on $J$ and $I$, respectively, and $f_{I \times J}$ is the average of $f$ on $I \times J$. $BMO_{Rect}(\mathbb{T}^2)$ is one of several possible generalizations of $BMO(\mathbb{T})$ to the two-variable case. We focus on this particular one because it fits naturally into our scheme. A treatment of rectangular $BMO$ can be found in Ferguson and Sadosky \cite{fergsad}. 

To obtain $BMO_{Rect}(\mathbb{T}^2)$ via our construction, let $X = L^2(\mathbb{T}^2)/\C$, $Y = L^2(\mathbb{T}^2)$ and
\begin{equation*}
\L = \left\{ L_{I,J} \, : \, L_{I,J}f = \chi_{I \times J} \frac{1}{|I||J|}(f - f_J - f_I + f_{I\times J}) \right\},
\end{equation*}
where $I$ and $J$ range over all non-empty arcs. Denoting by $\tau$ the quotient topology considered in Example \ref{bmoex}, we equip $\L$ with the corresponding product topology $\tau \times \tau$, so that $L_{I,J} \to \infty$ means that $\min(|I|, |J|) \to 0$. By construction $M(X,\L) = BMO_{Rect}(\mathbb{T}^2)$. Accordingly, $M_0(X,\L)$ will be named $VMO_{Rect}(\mathbb{T}^2)$. Assumption B is verified exactly as in Example \ref{bmoexpt2}, letting $f_n = P_{1-1/n}(\zeta) * P_{1-1/n} (\lambda) * f$ be a double Poisson integral. 

From Theorem \ref{thm1} we hence obtain that $VMO_{Rect}(\mathbb{T}^2)^{**}$ is isometrically isomorphic to $BMO_{Rect}(\mathbb{T}^2)$ via the $L^2(\mathbb{T}^2)$-pairing, and Theorem \ref{thm2} gives
\begin{equation*}
\dist{f,VMO_{Rect}}_{BMO_{Rect}} = \varlimsup_{\min(|I|, |J|) \to 0} \norm{L_{I,J} f}_{L^2(\mathbb{T}^2)}.
\end{equation*}
Another possible generalization of $BMO(\mathbb{T})$ to several variables is known as product $BMO$. In \cite{LTW}, Lacey, Terwilleger and Wick explore the corresponding product $VMO$ space. It would be interesting to apply our techniques also to this case, but one meets the difficulty of defining a reasonable topology on the collection of all open subsets of $\R^n$. On the other hand, the predual of $BMO_{Rect}(\mathbb{T}^2)$ is given as a space spanned by certain ``rectangular atoms'' (see \cite{fergsad}) and is as such more difficult to understand than the predual of product $BMO$, which is the Hardy space $H^1(\mathbb{T}^n)$ of the $n$-torus. 
\end{example}

\begin{example} \label{lipex}
Let $0 < \alpha \leq 1$ and let $\Omega$ be a compact subset of $\R^n$. In this example we shall treat the Lipschitz-H\"{o}lder space $\text{Lip}_\alpha( \Omega )$. By definition, a real-valued function $f$ on $\Omega$ is in $\text{Lip}_\alpha(\Omega)$ if and only if 
\begin{equation*}
\norm{f}_{\text{Lip}_\alpha(\Omega)} = \sup_{\substack{x,y \in \Omega \\ x \neq y}} \frac{|f(x)-f(y)|}{|x-y|^\alpha} < \infty.
\end{equation*}
As usual we identify $f$ and $f+C$, $C$ constant, in order to obtain a norm.

$X$ will be chosen as a quotient space of an appropriate fractional Sobolev space (Besov space) $W^{l,p}(\R^n)$. For $0 < l < 1$ and $1 < p < \infty$, $W^{l,p}$ consists of those $f \in L^p(\R^n)$ such that 
\begin{equation*}
\int_{\R^n} \int_{\R^n} \frac{|f(x)-f(y)|^p}{|x-y|^{pl+n}} \, \ud x \, \ud y < \infty.
\end{equation*}
Choose $l$ and $p$ such that $0 < l < \alpha$ and $pl > n$. By a Sobolev type embedding theorem (\cite{mazya}, Proposition 4.2.5) it then holds that $W^{l,p}$ continuously embeds into the space of continuous bounded functions $C_b(\R^n)$. Let
\begin{equation*}
A_\Omega = \{f \in W^{l,p} \, : \, f(x) = 0, \, x \in \Omega\}.
\end{equation*}
We set $X = W^{l,p}/A_\Omega$.

To obtain $\text{Lip}_\alpha(\Omega)$ through our construction, let $Y = \C$ and let every operator $L_{x,y} \in \L$, $x,y \in \Omega$, $x \neq y$, be of the form
\begin{equation*}
L_{x,y} f = \frac{|f(x)-f(y)|}{|x-y|^\alpha}.
\end{equation*}
We give $\L$ the topology of $\{(x,y) \in \Omega \times \Omega \, : \, x\neq y \}$. It then holds that $M(X,\L) = \text{Lip}_\alpha(\Omega)$. One inclusion is obvious and to see the other one let $f \in \text{Lip}_\alpha(\Omega)$. As in \cite{lipext}, $f$ can be extended to $\hat{f} \in \text{Lip}_\alpha(\R^n)$, $\hat{f} = f$ on $\Omega$. Letting $\chi \in C_c^\infty(\R^n)$ be a cut-off function such that $\chi(x) = 1$ for $x \in \Omega$, it is straightforward to check that $\chi \hat{f} \in M(X,\L)$, verifying that $\text{Lip}_\alpha(\Omega) \subset M(X,\L)$.

Note that $M_0(X,\L) = \text{lip}_\alpha(\Omega)$ is the corresponding little H\"{o}lder space, consisting of those $f \in \text{Lip}_\alpha(\Omega)$ such that
\begin{equation*}
\varlimsup_{|x-y| \to 0} \frac{|f(x)-f(y)|}{|x-y|^\alpha} = 0.
\end{equation*}
When $\alpha = 1$, the space $\text{lip}_\alpha(\Omega)$ is empty in many cases, so that Assumption \ref{as1} may fail. However, for $\alpha < 1$ we can verify Assumption \ref{as2} in general. Let $P_t$, $t > 0$, be the $n$-dimensional Poisson kernel,
\begin{equation*}
P_t(x) = c\frac{t}{(t^2+|x|^2)^{(n+1)/2}}, \quad x \in \R^n,
\end{equation*}
for the appropriate normalization constant $c$, and let $\chi$ denote the same cut-off as before. For $f \in M(X,\L)$ it is straightforward to verify that $\chi \cdot (P_t * f) \in M_0(X,\L)$,
\begin{equation*}
\norm{\chi \cdot(P_t * f)}_{\text{Lip}_\alpha(\Omega)} \leq \norm{f}_{\text{Lip}_\alpha(\Omega)},
\end{equation*}
and that $\chi \cdot (P_t * f) \to f = \chi f$ weakly in $X$ as $t \to 0^+$, the final statement following from the reflexivity of $X$ and the fact that $P_t * f$ tends to $f$ pointwise almost everywhere.

We conclude that, for $0 < \alpha < 1$, 
\begin{equation*}
\dist{f, \text{lip}_\alpha(\Omega)}_{\text{Lip}_\alpha(\Omega)} = \varlimsup_{|x-y| \to 0} \frac{|f(x)-f(y)|}{|x-y|^\alpha}.
\end{equation*}
In addition, Theorem \ref{thm1} says that $\text{lip}_\alpha(\Omega)^{**} \simeq \text{Lip}_\alpha(\Omega)$ isometrically. In Kalton \cite{kalton} it is proven that $\text{lip}_\alpha(M)^{**} \simeq \text{Lip}_\alpha(M)$ under very general conditions on $M$, for example whenever $M$ is a compact metric space. It would be interesting to see if the theorems of this paper can be applied in this situation and, if this is the case, which space $X$ to embed $\text{Lip}_\alpha(M)$ into.
\end{example}

\subsection*{Acknowledgements}
The author is grateful to Alexandru Aleman and Sandra Pott for helpful discussions, ideas and support regarding this paper, and to Jos\'{e} Bonet for pointing out the application considered in Example \ref{weightedex}.


\begin{thebibliography}{99}
\bibitem{AS}
Aleman, A., Simbotin, A-M, Estimates in M\"{o}bius invariant spaces of analytic functions, Complex Var. Theory Appl. \textbf{49} (2004), no. 7-9, 487--510.

\bibitem{andersonduncan}
Anderson, J. M., Duncan, J., Duals of Banach spaces of entire functions,
Glasgow Math. J. \textbf{32} (1990), no. 2, 215--220. 

\bibitem{AFP}
Arazy J., Fisher S. D., Peetre J, M\"{o}bius invariant function spaces, J. Reine Angew. Math. \textbf{363} (1985), 110--145.

\bibitem{attele}
Attele, K. R. M., Interpolating sequences for the derivatives of Bloch functions, Glasgow Math. J. \textbf{34} (1992), no. 1, 35--41.

\bibitem{axlershapiro}
Axler, S., Shapiro, J. H., Putnam's theorem, Alexander's spectral area estimate, and VMO, Math. Ann. \textbf{271} (1985), no. 2, 161--183. 

\bibitem{bierstedt}
Bierstedt, K. D., Summers, W. H., Biduals of weighted Banach spaces of analytic functions, J. Austral. Math. Soc. Ser. A \textbf{54} (1993), no. 1, 70--79.

\bibitem{bierstedtbonet}
Bierstedt, K. D., Bonet, J., Galbis, A., Weighted spaces of holomorphic functions on balanced domains, Michigan Math. J. \textbf{40} (1993), no. 2, 271--297.

\bibitem{carmonacufi}
Carmona, J., Cuf\'{i}, J.,, On the distance of an analytic function to VMO, J. London Math. Soc. (2) \textbf{34} (1986), no. 1, 52--66.

\bibitem{conway}
Conway, J. B, \textit{``A course in functional analysis''}, Second edition, Graduate Texts in Mathematics, \textbf{96}. Springer-Verlag, New York, 1990.

\bibitem{lipext}
Czipszer, J., Geh\'{e}r, L, Extension of functions satisfying a Lipschitz condition, Acta Math. Acad. Sci. Hungar. \textbf{6} (1955), 213--220.

\bibitem{dobrakov1}
Dobrakov, I., On integration in Banach spaces. I, Czechoslovak Math. J. \textbf{20(95)} (1970), 511--536. 

\bibitem{dobrakov2}
Dobrakov, I., On integration in Banach spaces. II, Czechoslovak Math. J. \textbf{20(95)} (1970), 680--695.

\bibitem{dobrakov3}
Dobrakov, I., On representation of linear operators on $C_{0}\,(T,\,{\rm {\bf }X})$, Czechoslovak Math. J. \textbf{21 (96)} 1971 13--30.

\bibitem{EssenWulan}
Ess\'{e}n, M., Wulan, H., On analytic and meromorphic functions and spaces of $Q_K$-type, Illinois J. Math. \textbf{46} (2002), no. 4, 1233--1258.

\bibitem{fergsad}
Ferguson, S. H., Sadosky, C., Characterizations of bounded mean oscillation on the polydisk in terms of Hankel operators and Carleson measures, J. Anal. Math. \textbf{81} (2000), 239--267.

\bibitem{garcia}
Garc\'{i}a, D., Maestre, M., Rueda, P, Weighted spaces of holomorphic functions on Banach spaces, Studia Math. \textbf{138} (2000), no. 1, 1--24.

\bibitem{garnett}
Garnett, J. B., \textit{Bounded analytic functions. Revised first edition}, Graduate Texts in Mathematics, \textbf{236}. Springer, New York, 2007

\bibitem{godefroy}
Godefroy, G., Existence and uniqueness of isometric preduals: a survey, \textit{``Banach space theory''} (Iowa City, IA, 1987), 131--193, Contemp. Math., \textbf{85}, Amer. Math. Soc., Providence, RI, 1989. 

\bibitem{halmos}
Halmos, P. R., \textit{``Measure Theory''}, D. Van Nostrand Company, Inc., New York, N. Y., 1950.

\bibitem{kalton}
Kalton, N. J., Spaces of Lipschitz and Hölder functions and their applications, Collect. Math. \textbf{55} (2004), no. 2, 171--217.

\bibitem{LTW}
Lacey, M. T., Terwilleger, E.; Wick, B. D., Remarks on product VMO, Proc. Amer. Math. Soc. \textbf{134} (2006), no. 2, 465--474 (electronic).

\bibitem{mazya}
Maz'ya, V. G., Shaposhnikova, T. O., \textit{``Theory of Sobolev multipliers. With applications to differential and integral operators''}, Grundlehren der Mathematischen Wissenschaften [Fundamental Principles of Mathematical Sciences], \textbf{337}. Springer-Verlag, Berlin, 2009.

\bibitem{nara}
Nara, C., Uniqueness of the predual of the Bloch space and its strongly exposed points, Illinois J. Math. \textbf{34} (1990), no. 1, 98--107. 

\bibitem{rubelshields}
Rubel, L. A., Shields, A. L., The second duals of certain spaces of analytic functions, J. Austral. Math. Soc. \textbf{11} 1970 276--280.

%\bibitem{shapiroshieldstaylor}
%Shapiro, J. H., Shields A. L., Taylor, G. D., The second duals of some function spaces, unpublished manuscript, 1970/1971.

\bibitem{ShieldsWilliams}
Shields, A. L., Williams, D. L., Bounded projections, duality, and multipliers in spaces of analytic functions, Trans. Amer. Math. Soc. \textbf{162} 1971 287--302.

\bibitem{stegengasteph}
Stegenga, D. A., Stephenson, K., Sharp geometric estimates of the distance to VMOA, \textit{The Madison Symposium on Complex Analysis (Madison, WI, 1991)}, 421--432, Contemp. Math., \textbf{137}, Amer. Math. Soc., Providence, RI, 1992.

\bibitem{tjani}
Tjani, M., Distance of a Bloch function to the little Bloch space, Bull. Austral. Math. Soc. \textbf{74} (2006), no. 1, 101--119. 

\bibitem{zinger}
Zinger, I., Linear functionals on the space of continuous mappings of a compact Hausdorff space into a Banach space, (Russian) Rev. Math. Pures Appl. \textbf{2} 1957 301--315
\end{thebibliography}
\end{document}